\documentclass[12pt,leqno,oneside]{amsart}
\usepackage{mathrsfs,dsfont}
\usepackage{amsmath,amstext,amsthm,amssymb,amscd,bbm}
\usepackage{charter}
\usepackage{typearea}
\usepackage{pdfsync}
\usepackage{color}
\usepackage{mathtools}

\usepackage{hyperref}

\usepackage[width=6.4in,height=8.5in]
{geometry}

\numberwithin{equation}{section}

\newtheorem{Theorem}{Theorem}[section]

\newtheorem{Lemma}[Theorem]{Lemma}

\theoremstyle{definition}

\newcommand{\db}{\overline\partial}

\newcommand{\wi}{\widetilde}

\DeclareMathOperator{\ric}{Ric}
\DeclareMathOperator{\codim}{codim}

\DeclareMathOperator{\inte}{int}

\newcommand{\cali}[1]{\mathscr{#1}}
\newcommand{\cO}{\cali{O}}

\newcommand{\cT}{\cali{T}}
\newcommand{\cC}{\cali{C}}

\newcommand{\field}[1]{\mathbb{#1}}

\newcommand{\C}{\field{C}}
\newcommand{\N}{\field{N}}

\newcommand{\reg}{\mathrm{reg}}
\newcommand{\sing}{\mathrm{sing}}
\newcommand{\FS}{\mathrm{FS}}
\newcommand{\PSH}{\mathrm{PSH}}

\newcommand{\comment}[1]{}

\begin{document}

\title[Tian's theorem for Moishezon spaces]{Tian's theorem for Moishezon spaces}

\author{Dan Coman}
\thanks{D.\ Coman is supported by the NSF Grant DMS-2154273}
\address{Department of Mathematics, Syracuse University, 
Syracuse, NY 13244-1150, USA}
\email{dcoman@syr.edu}

\author{Xiaonan Ma}
\address{Universit\'e Paris Cit\'e, CNRS, IMJ-PRG, B\^atiment Sophie Germain, 
UFR de Math\'ematiques, 
\newline\mbox{\quad}\,Case 7012, 75205 Paris Cedex 13, France}
\email{xiaonan.ma@imj-prg.fr}
\thanks{X.\ Ma is partially supported by NSFC No.11829102, ANR-21-CE40-0016 and
funded through the Institutional Strategy of the University of Cologne within the German 
Excellence Initiative}

\author{George Marinescu}
\address{Univerisit\"at zu K\"oln, Mathematisches institut,
Weyertal 86-90, 50931 K\"oln, Germany 
\newline\mbox{\quad}\,Institute of Mathematics `Simion Stoilow', 
Romanian Academy, Bucharest, Romania}
\email{gmarines@math.uni-koeln.de}
\thanks{G.\ Marinescu is partially supported
by the DFG funded project
CRC TRR 191 \textquoteleft Symplectic Structures in Geometry, Algebra
and Dynamics\textquoteright{} (Project-ID 281071066-TRR 191),
the DFG Priority Program 2265 \textquoteleft Random
Geometric Systems\textquoteright{} (Project-ID 422743078) and the
ANR-DFG project QuaSiDy (Project-ID 490843120).}

\date{December 16, 2024}

\pagestyle{myheadings}

\begin{abstract} 
We prove that the Fubini-Study currents associated to a sequence 
of singular Hermitian holomorphic line bundles on a compact normal Moishezon space
distribute asymptotically as the curvature currents of their metrics.

\end{abstract}

\maketitle
\tableofcontents

\section{Introduction}\label{S:intro}

Let $(L,h)$ be a positive Hermitian holomorphic line bundle on a projective manifold $X$
and set $(L^p,h^p)=(L^{\otimes p},h^{\otimes p})$. Kodaira's embedding theorem states that 
for all $p$ sufficiently large, the Kodaira map $\Phi_p:X\to\mathbb P\big(H^0(X,L^p)^\star\big)$ 
associated to $(L^p,h^p)$ is an embedding. Hence one can consider the Fubini-Study forms on $X$, 
$\gamma_p=\Phi_p^\star(\omega_\FS)$, where $\omega_\FS$ denotes the Fubini-Study form 
on a projective space. A celebrated theorem of Tian \cite{Ti90} shows that 
$\frac{1}{p}\,\gamma_p\to c_1(L,h)$ as $p\to\infty$, in the $\cC^\infty$ topology on $X$ (see also \cite{Ru98}). 
Tian's theorem follows from the first term asymptotics of the Bergman kernel function associated to the space 
$H^0(X,L^p)$ endowed with the inner product determined by $h^p$ and a volume form on $X$. 
We refer to the book \cite{MM07} for an exposition of these topics as well as for the full asymptotic expansion 
of the Bergman kernel in different contexts.

In \cite{CM15} we extended Tian's theorem to the case when
$(L,h)$ is a singular Hermitian holomorphic line bundle with strictly 
positive curvature current on a compact K\"ahler manifold $X$, 
the above convergence now being in the weak sense of currents. 
Later, we extended Tian's theorem further to general classes of 
compact K\"ahler spaces $X$ \cite{CM13,CMM17}. 
In all these situations one has to replace the space $H^0(X,L^p)$ 
with the Bergman space $H^0_{(2)}(X,L^p)$ of square integrable 
holomorphic sections.

In \cite[Theorem 1.1]{CMM17} we generalized Tian's theorem 
by considering sequences $(L_p,h_p)$, $p\geq1$, 
of singular Hermitian holomorphic line bundles over 
a compact normal K\"ahler space $X$, in place of the 
sequence of powers $(L^p,h^p)$ of a line bundle $(L,h)$. 
Assuming that the curvature currents $c_1(L_p,h_p)$ satisfy 
a natural growth condition, we proved that the  Fubini-Study currents 
$\gamma_p$ associated to the Bergman spaces $H^0_{(2)}(X,L_p)$ 
(see \eqref{e:BFS1}) distribute asymptotically like $c_1(L_p,h_p)$.

The purpose of this note is to show that the preceding
result holds more generally for compact normal spaces $X$ 
which are not assumed to be K\"ahler. 
The precise setting is the following: 

\smallskip

(A) $X$ is a compact, reduced, irreducible, normal complex space of dimension $n$, 
$X_\reg$ denotes the set of regular points of $X$, $X_\sing$ denotes the set 
of singular points of $X$, and $\omega$ is a Hermitian form on $X$.

\smallskip

(B) $(L_p,h_p)$, $p\geq1$, is a sequence of holomorphic line 
bundles on $X$ with singular Hermitian metrics $h_p$ whose 
curvature currents verify 
\begin{equation}\label{e:pc}
c_1(L_p,h_p)\geq a_p\,\omega \, \text{ on $X$, where 
$a_p>0$ and } \lim_{p\to\infty}a_p=\infty.
\end{equation}
We let $A_p=\int_Xc_1(L_p,h_p)\wedge\omega^{n-1}$ and assume that 
\begin{equation}\label{e:domin0}
\exists\,T_0\in\cT(X) \text{ such that } 
c_1(L_p,h_p)\leq A_pT_0\,,\;\forall\,p\geq1\,.
\end{equation}
Condition (B) implies that $L_p$
are big line bundles, hence $X$ is a Moishezon space. 

\smallskip

Let $d^c:= \frac{1}{2\pi i}\,(\partial -\overline\partial)$, so 
$dd^c=\frac{i}{\pi}\,\partial\overline\partial$. We consider currents on $X$
in the sense of \cite{D85}, and denote by $\cT(X)$ the set of positive closed currents of 
bidegree $(1,1)$ on $X$ which have local plurisubharmonic (psh) potentials, i.e., 
$T=dd^cv$ holds in a neighborhood of each point of $X$ for some psh function $v$. 
We refer to \cite[Section 2.1]{CMM17}) for a review of the notions of differential forms, 
psh functions and currents on complex spaces. We denote by $\PSH(U)$ the set of psh 
functions on an open set $U\subset X$. The notions of singular Hermitian metric on a line 
bundle over a complex space $X$, and its curvature current, are defined as in the case 
when $X$ is smooth (see \cite{D90}, \cite[Section 2.2]{CMM17}).

\smallskip

Let $H^0_{(2)}(X,L_p)$ be the Bergman space 
of $L^2$-holomorphic sections of $L_p$ relative to the metric 
$h_p$ and the volume form induced by $\omega$ on $X$, 
\begin{equation}\label{e:bs}
H^0_{(2)}(X,L_p)=H^0_{(2)}(X,L_p,h_p,\omega^n)=
\Big\{S\in H^0(X,L_p):\,\|S\|_p^2:=\int_X|S|^2_{h_p}\,\frac{\omega^n}{n!}<\infty\Big\},
\end{equation}
endowed with the obvious inner product. Let $P_p,\gamma_p$ be the 
Bergman kernel function and the Fubini-Study current of the space $H^0_{(2)}(X,L_p)$.
They are defined as follows. Let $S_1^p,\dots,S_{d_p}^p$ be an orthonormal 
basis of $H^0_{(2)}(X,L_p)$. If $x\in X$ let $e_p$ be a 
holomorphic frame of $L_p$ on a neighborhood $U_p$ of $x$ 
and write $S_j^p=s_j^pe_p$ with $s_j^p\in\cO_X(U_p)$. Then 
\begin{equation}\label{e:BFS1}
P_p(x)=\sum_{j=1}^{d_p}|S^p_j(x)|_{h_p}^2\,,\,\;
\gamma_p\vert_{U_p}=\frac{1}{2}\,dd^c
\log\left(\sum_{j=1}^{d_p}|s_j^p|^2\right).
\end{equation}
We have that $P_p,\,\gamma_p$ are independent of the choice 
of basis. Moreover, $\gamma_p=\Phi_p^\star(\omega_\FS)$, where 
$\Phi_p:X\dashrightarrow\mathbb P\big(H^0_{(2)}(X,L_p)^\star\big)$ is the (meromorphic)
Kodaira map associated to the Bergman space $H^0_{(2)}(X,L_p)$.

Our main result is the following theorem:

\begin{Theorem}\label{T:mt1} 
Assume that $X,\omega,(L_p,h_p),\,p\geq1$, satisfy 
conditions (A)-(B). Then the following hold:

\smallskip

\par (i) $\frac{1}{A_p}\log P_p\to 0$ as $p\to\infty$,
in $L^1(X,\omega^n)$. 

\smallskip

\par (ii) $\frac{1}{A_p}\,(\gamma_p-c_1(L_p,h_p))\to 0$ as 
$p\to\infty$, in the weak sense of currents on $X$.

\comment{
\smallskip

Moreover, if $\frac{1}{A_p}\,c_1(L_p,h_p)\to T$ for some positive closed current $T$ of bidegree $(1,1)$ 
on $X$, then $\frac{1}{A_p}\,\gamma_p\to T$ as $p\to\infty$, in the weak sense of currents on $X$.
}
\end{Theorem}

Note that a complex space $X$ that verifies (A)-(B) is a {\em Moishezon space}. 
Thus Theorem \ref{T:mt1} applies to any compact normal Moishezon space $X$, 
which is not necessarily assumed to be K\"ahler. 
Indeed, a singular Hermitian holomorphic line bundle $(L_p,h_p)$ over $X$ 
with strictly positive curvature current as in \eqref{e:pc} is big, 
hence $X$ is Moishezon (see, e.g., \cite[Proposition 2.3]{CMM17}, \cite[Propositions 3.2 and 3.3]{BCMN}). 
We recall that a (reduced) compact irreducible complex space $X$ of dimension $n$ is called a Moishezon space
if there exist $n$ algebraically independent meromorphic functions on $X$ 
(see \cite[Definition 3.5]{U75}, \cite[Section 3]{BCMN}). We refer to \cite[Section 3]{BCMN} 
and the references therein for the definition and some basic properties of big line bundles over complex spaces.

Theorem \ref{T:mt1} is proved in Section \ref{S:pfmt}. An important special case is provided 
by the sequence of powers $(L_p,h_p)=(L^p,h^p)$ of a singular Hermitian holomorphic line bundle $(L,h)$ 
with strictly positive curvature current. See Theorem \ref{T:mt2} in Section \ref{S:appl}, which gives a 
full generalization of Tian's theorem to the singular setting. We recall in Section \ref{S:appl} 
a few other important applications of Theorem \ref{T:mt1}, in particular to the asymptotic distribution 
of the zeros of random sequences of holomorphic sections (see Theorem \ref{T:zeros}).

\section{Proof of Theorem \ref{T:mt1}}\label{S:pfmt}

By a theorem of Moishezon (\cite{Moi66}, \cite[Theorem 3.6]{U75}), 
$X$ is bimeromorphically equivalent to a projective manifold. 
More precisely, since $X$ is assumed to be normal we have that 
$\codim X_\sing\geq2$ and the following holds 
(see \cite[Theorem 3.1]{BCMN}):

\begin{Theorem}\label{T:Moi}
If $X$ is a compact, irreducible, normal Moishezon space then there exists a connected 
projective manifold $\wi X$ and a surjective holomorphic map $\pi:\wi X\to X$, 
given as a composition of finitely many blow-ups with smooth center, such that 
$\pi:\wi X\setminus\Sigma\to X\setminus Y$ is a biholomorphism, where $Y$ is an 
analytic subset of $X$, $\codim Y\geq2$, $X_\sing\subset Y$, 
and $\Sigma=\pi^{-1}(Y)$ is a normal crossings divisor. 
\end{Theorem}

Let $X,\omega, (L_p,h_p)$ verify assumptions (A)-(B) and $\pi:\wi X\to X$ be as in 
Theorem \ref{T:Moi}. 
In \cite{CMM17} we assumed that $X$ is a normal K\"ahler space, 
and we showed that the desingularization $\wi{X}$ obtained by
finitely many blow-ups with smooth centers as in \cite{BM97, GM06} is K\"ahler.
This is crucial for construction peak sections by using methods involving $\db$.
In our present situation we obtain a projective desingularization $\wi{X}$
since $X$ is Moishezon.

We will follow the arguments from the proof of 
\cite[Theorem 1.1]{CMM17}, working with $\pi:\wi X\to X$ instead of the desingularization 
of $X$ given in \cite[Section 2.3]{CMM17}, and using 
a K\"ahler form $\wi\omega$ on the projective manifold $\wi X$. 
We recall the following 
lemmas that will be needed in the proof.

\begin{Lemma}[{\cite[Lemma 2.1]{CMM17}}]\label{L:bPp} If
\[H^0_{(2)}(\widetilde X,\pi^\star L_p)=
\left\{\widetilde S\in H^0(\widetilde X,\pi^\star L_p):\,
\int_{\widetilde X}|\widetilde S|^2_{\pi^\star h_p}\,
\frac{\pi^\star\omega^n}{n!}<\infty\right\},\]
the map 
    $\pi^\star:H^0_{(2)}(X,L_p)\longrightarrow 
    H^0_{(2)}(\widetilde X,\pi^\star L_p)$ is an isometry and 
    the Bergman kernel function of 
    $H^0_{(2)}(\widetilde X,\pi^\star L_p)$ is 
    $\widetilde P_p=P_p\circ\pi$. 
\end{Lemma}

\comment{
\begin{Lemma}{\cite[Lemma 2.2]{CMM17}}\label{L:Moi} 
Let $X,\omega$ be as in Theorem \ref{T:mt1} and 
$\pi:\wi X\to X$ be as in Theorem \ref{T:Moi}. 
Then there exists a smooth Hermitian metric $\theta$ on 
$F=\cO_{\widetilde X}(-\Sigma)$ and $C>0$ such that 
$\Omega=C\pi^\star\omega+c_1(F,\theta)$ is a Hermitian form on 
$\widetilde X$ and $\Omega\geq\pi^\star\omega$. 
\end{Lemma}

\begin{Lemma}[{\cite[Lemma 3.2]{CMM17}}]\label{L:tildehp} 
Let $X,\omega, (L_p,h_p)$ verify 
assumptions (A)-(B), and let $F=\cO_{\wi X}(-\Sigma)$, 
$\theta$, $\Omega=C\pi^\star\omega+c_1(F,\theta)$ be as in 
Lemma \ref{L:Moi}. Then there exist $\alpha\in(0,1)$, 
$b_p\in\mathbb N$, 
and singular Hermitian metrics $\wi h_p$ on 
$\pi^\star L_p\vert_{\wi X\setminus\Sigma}$ 
such that $a_p\geq Cb_p\,$, $b_p\to\infty$ and $b_p/A_p\to0$ 
as $p\to\infty$, $\wi h_p\geq\alpha^{b_p}\pi^\star h_p$ and 
 $c_1(\pi^\star L_p,\wi h_p)\geq b_p\Omega$ on
 $\wi X\setminus\Sigma$. 
 Moreover, for every relatively compact open subset $\wi U$ 
 of $\wi X\setminus\Sigma$ there exists a constant 
 $\beta_{\wi U}>1$ such that 
 $\wi h_p\leq\beta_{\wi U}^{b_p}\pi^\star h_p$ on $\wi U$.  
\end{Lemma}
}

\begin{Lemma}[{\cite[Lemma 2.2,\,Lemma 3.2]{CMM17}}]\label{L:thp}  
There exist $\alpha\in(0,1)$, $b_p\in\mathbb N$, 
a Hermitian form $\Omega$ on $\wi X$,
and singular Hermitian metrics $\wi h_p$ on 
$\pi^\star L_p\vert_{\wi X\setminus\Sigma}$ 
such that $\Omega\geq\pi^\star\omega$, $b_p\to\infty$ and $b_p/A_p\to0$ 
as $p\to\infty$, $\wi h_p\geq\alpha^{b_p}\pi^\star h_p$ and 
$c_1(\pi^\star L_p,\wi h_p)\geq b_p\Omega$ on $\wi X\setminus\Sigma$. 
Moreover, for every relatively compact open subset $\wi U$ of $\wi X\setminus\Sigma$ 
there exists a constant $\beta_{\wi U}>1$ such that 
$\wi h_p\leq\beta_{\wi U}^{b_p}\pi^\star h_p$ on $\wi U$.  
\end{Lemma}

The Hermitian form $\Omega$ is obtained as $\Omega=C\pi^\star\omega+c_1(F,\theta)$,
where $\theta$ is a suitable metric on $F=\cO_{\widetilde X}(-\Sigma)$ and $C>0$ is
an appropriate constant. If $b_p\in\mathbb N$ is a sequence 
such that $b_p\to\infty$, $a_p\geq Cb_p$, $b_p/A_p\to0$ and if 
$\varphi$ is a weight  of $\theta$ on $\wi{X}\setminus\Sigma$, one defines the
metric $\wi{h}_p=e^{-2b_p\varphi}\pi^\star h_p$ on $\pi^\star L_p\vert_{\wi X\setminus\Sigma}$
and shows that it has the desired properties.
In particular the positivity of $c_1(\pi^\star L_p,\wi{h}_p)$ is needed to solve a 
$\db$-equation on $\wi{X}\setminus\Sigma$, by using the following version of Demailly's estimates 
for the $\db$-operator
\cite[Th\'eor\`eme 5.1]{D82} (see also \cite[Theorem 2.5]{CMM17}):

\begin{Theorem}\label{T:db}
Let $Z$, $\dim Z=n$, be a complete K\"ahler manifold and $\Theta$ be a 
K\"ahler form on $Z$ (not necessarily complete) such that its Ricci form 
$\ric_\Theta\geq-2\pi B\Theta$ for some constant $B>0$. Let 
$(L_p,h_p)$ be singular Hermitian holomorphic line bundles on $Z$ such that 
$c_1(L_p,h_p)\geq2a_p\Theta$, where $a_p\geq B$. If $g\in L^2_{0,1}(Z,L_p,loc)$ 
verifies $\db g=0$ and $\int_Z|g|^2_{h_p}\,\Theta^n<\infty$ then there exists 
$u\in L^2_{0,0}(Z,L_p,loc)$ such that $\db u=g$ and 
$\int_Z|u|^2_{h_p}\,\Theta^n\leq\frac{1}{a_p}\, \int_Z|g|^2_{h_p}\,\Theta^n$. 
\end{Theorem}

\begin{proof}[Proof of Theorem \ref{T:mt1}] 
By \eqref{e:BFS1} we have that $\log P_p\in L^1(X,\omega^n)$ and 
\[\gamma_p-c_1(L_p,h_p)=\frac{1}{2}\,dd^c\log P_p.\]
Thus $(ii)$ follows at once from $(i)$. To prove $(i)$ we proceed in two steps. 

\smallskip

\par {\em Step 1.} We prove that 
$\frac{1}{A_p}\log P_p\to0$ as $p\to\infty$, in 
$L^1_{loc}(X\setminus Y,\omega^n)$. Fix $x\in X\setminus Y\subset X_\reg$, 
$W\Subset X\setminus Y$ a contractible Stein coordinate 
neighborhood of $x$, $r_0>0$ such that the (closed) ball 
$V:=B(x,2r_0)\subset W$, and set $U=B(x,r_0)$. 
Note that the currents $\{\frac{1}{A_p}\,c_1(L_p,h_p)\}$ 
have uniformly bounded mass. By \cite[Proposition A.16]{DS10} 
(see also \cite{DNS08}) and \cite[Theorem 3.2.12]{Ho}, 
we infer that there exist psh functions $\psi_p$ on $\inte V$ such that 
$dd^c\psi_p=c_1(L_p,h_p)$ and the sequence $\{\frac{1}{A_p}\,\psi_p\}$
is relatively compact in $L^1_{loc}(\inte V,\omega^n)$. 
Since $L_p|_W$ is holomorphically trivial, we can find holomorphic 
frames $e_p$ for $L_p|_{\inte V}$ such that $\psi_p$ 
are the corresponding psh weights of $h_p$, so $|e_p|_{h_p}=e^{-\psi_p}$.

Let $\wi\omega$ be a K\"ahler form on $\wi X$ and $\Omega$
be the Hermitian form from Lemma \ref{L:thp}. 
Then there exists constants $\delta_1,\delta_2>0$ such that 
\begin{equation}\label{e:KH}
\Omega\geq\delta_1\wi\omega\,,\;\wi\omega\geq\delta_2\Omega\geq
\delta_2\pi^\star\omega.
\end{equation} 

With $\{b_p\}$ as in Lemma \ref{L:thp}, we prove that there exist 
$C_1>1$ and $p_0\in\mathbb N$ such that 
\begin{equation}\label{e:Bkf}
-\frac{b_p\log C_1}{A_p}\leq\frac{\log P_p(z)}{A_p}
\leq\frac{\log(C_1r^{-2n})}{A_p}+\frac{2}{A_p}\,\left(\max_{B(z,r)}
\psi_p-\psi_p(z)\right)
\end{equation}
holds for all $p>p_0$, $0<r<r_0$, and $z\in U$ with $\psi_p(z)>-\infty$. 
The upper bound in \eqref{e:Bkf} follows from the subaverage inequality, 
exactly as the upper bound from \cite[(7)]{CM15}. 

We show next that there exist $c\in(0,1)$ and $p_0\in\mathbb N$ with the following 
property: if $p>p_0$ and $z\in U$ is such that $\psi_p(z)>-\infty$, then there exists 
$S_{z,p}\in H^0_{(2)}(X,L_p)$ with $S_{z,p}(z)\neq0$ and 
\begin{equation}\label{e:lest}
c^{b_p}\|S_{z,p}\|_p^2\leq|S_{z,p}(z)|^2_{h_p}.
\end{equation}
This yields the lower bound in \eqref{e:Bkf}, since 
$P_p(z)\geq|S_{z,p}(z)|^2_{h_p}/\|S_{z,p}\|_p^2\geq c^{b_p}$. 
To this end we work first on 
$\pi^\star L_p|_{\wi X\setminus\Sigma}$ using the metric 
$\wi h_p$ from Lemma \ref{L:thp}. 
By estimates \eqref{e:KH},
\[c_1\big(\pi^\star L_p\vert_{\wi X\setminus\Sigma},\wi h_p\big)
\geq b_p\Omega\geq\delta_1b_p\wi\omega \text{ on } \wi X\setminus\Sigma,
\text{ where } b_p\to\infty.\]

We have that $\wi X\setminus\Sigma$ has a complete K\"ahler 
metric \cite{D82,O87}, 
and $\ric_{\wi\omega}\geq-2\pi B\wi\omega$ on $\wi X$ for some $B>0$.
Using ideas from \cite[Proposition 3.1]{D92}, \cite[Section 9]{D93}, we 
apply the Ohsawa-Takegoshi extension theorem \cite{OT87} and 
Theorem \ref{T:db} as in the proof of \cite[Theorem 5.1]{CM15} to show 
that there exist $C_2>1$, $p_0\in\mathbb N$, such that if $p>p_0$ and 
$\wi z\in\pi^{-1}(U)$, $\psi_p\circ\pi(\wi z)>-\infty$, then there is 
$\wi S\in H^0(\wi X\setminus\Sigma,\pi^\star L_p)$ verifying $\wi S(\wi z)\neq0$ and 
\[\int_{\wi X\setminus\Sigma}|\wi S|^2_{\wi h_p}\,\frac{\wi\omega^n}{n!}\leq 
C_2|\wi S(\wi z)|^2_{\wi h_p}.\]
By Lemma \ref{L:thp} and \eqref{e:KH} we obtain  
\begin{equation}\label{e:lest1}
\delta_2^n\alpha^{b_p}\int_{\wi X\setminus\Sigma}|\wi S|^2_{\pi^\star h_p}\,
\frac{\pi^\star\omega^n}{n!}\leq C_2\beta^{b_p}|\wi S(\wi z)|^2_{\pi^\star h_p},
\end{equation}
where $\beta>1$ is so that $\wi h_p\leq\beta^{b_p}\pi^\star h_p$ on $\pi^{-1}(U)$.
As $\pi:\wi X\setminus\Sigma\to X\setminus Y$ is a biholomorphism, we let $z=\pi(\wi z)$ and $S_{z,p}$ be the section of $L_p\vert_{X\setminus Y}$
induced by $\wi S$. Since $X$ is normal and $\codim Y\geq2$, $S_{z,p}$ 
extends to a holomorphic section on $X$ and \eqref{e:lest} follows from \eqref{e:lest1}. 

Recall that $\{\frac{1}{A_p}\,\psi_p\}$ is relatively compact in 
$L^1_{loc}(\inte V,\omega^n)$, hence it is locally uniformly upper 
bounded in $\inte V$. It follows from \eqref{e:Bkf} that there is a 
constant $C_3>0$ such that 
\begin{equation}\label{e:uest}
\left|\frac{1}{A_p}\log P_p\right|
\leq C_3-\frac{2}{A_p}\,\psi_p\,\text{ a.e.\ on } U,\;\forall\,p>p_0.
\end{equation}
Moreover, if a subsequence $\frac{1}{A_{p_j}}\,\psi_{p_j}\to\psi$
in $L^1_{loc}(\inte V,\omega^n)$ and a.e.\ on $\inte V$, where $\psi$ is psh on $\inte V$, 
we infer from \eqref{e:Bkf} and the Hartogs lemma \cite[Theorem 3.2.13]{Ho} that 
\[0\leq\liminf\frac{\log P_{p_j}(z)}{A_{p_j}}
\leq\limsup\frac{\log P_{p_j}(z)}{A_{p_j}}
\leq2\left(\max_{B(z,r)}\psi-\psi(z)\right)\]
holds for a.e.\ $z\in U$ and every $r<r_0$. Thus 
$\frac{1}{A_{p_j}}\log P_{p_j}\to0$ a.e.\ on $U$, and hence 
in $L^1(U,\omega^n)$ by \eqref{e:uest} and the generalized Lebesgue dominated convergence theorem.
We conclude that $\frac{1}{A_p}\log P_p\to0$ as $p\to\infty$ in 
$L^1_{loc}(X\setminus Y,\omega^n)$.

\smallskip

{\em Step 2.} We finish the proof of $(i)$ by showing that 
there exists a compact set 
$K\subset X$ such that $Y\subset\inte K$ and 
$\frac{1}{A_p}\log P_p\to0$ in $L^1(K,\omega^n)$.
Let $H^0_{(2)}(\wi X,\pi^\star L_p)$ be the Bergman spaces from 
Lemma \ref{L:bPp}. It follows by \eqref{e:domin0}
that there exists $M>0$ such that 
\begin{equation}\label{e:domin}
\int_{\wi X}c_1(\pi^\star L_p,\pi^\star h_p)\wedge\Omega^{n-1}
\leq MA_p\;,\;\;\forall\,p\geq1\,.
\end{equation}

Let $y\in\Sigma$. By \eqref{e:domin}, we can proceed as in 
Step 1 to find an open neighborhood $\wi W$ of $y$ and 
holomorphic frames $\wi e_p$ of $\pi^\star L_p\vert_{\wi W}$ 
with corresponding psh weights $\wi\psi_p$ of $\pi^\star h_p$, 
such that the sequence 
$\{\frac{1}{A_p}\,\wi\psi_p\}$
is relatively compact in $L^1_{loc}(\wi W,\Omega^n)$. 
Let $\{\wi S^p_j:\,1\leq j\leq d_p\}$
be an orthonormal basis of $H^0_{(2)}(\wi X,\pi^\star L_p)$ 
and $\wi S^p_j=\wi s^p_j\wi e_p$\,, 
with $\wi s^p_j\in\cO_{\wi X}(\wi W)$. By Lemma \ref{L:bPp}, 
\[\frac{1}{A_p}\,\wi v_p-\frac{1}{A_p}\,
\wi\psi_p=\frac{1}{2A_p}\log P_p\circ\pi, \text{ where }
\wi v_p=\frac{1}{2}\,\log\Big(\sum_{j=1}^{d_p}|\wi s^p_j|^2\Big)\in\PSH(\wi W).\]

We claim that $\frac{1}{A_p}\log P_p\circ\pi\to0$ in 
$L^1_{loc}(\wi W,\Omega^n)$. Indeed, assume that a subsequence 
$\{\frac{1}{A_{p_j}}\,\wi\psi_{p_j}\}$ 
converges in $L^1_{loc}(\wi W,\Omega^n)$ to a psh function 
$\wi\psi$ on $\wi W$. By Step 1, $\frac{1}{A_p}\log P_p\circ\pi\to0$ in
$L^1_{loc}(\wi W\setminus\Sigma,\Omega^n)$, hence $\frac{1}{A_{p_j}}\,\wi v_{p_j}\to\wi\psi$ in 
$L^1_{loc}(\wi W\setminus\Sigma,\Omega^n)$. It follows that $\{\frac{1}{A_{p_j}}\,\wi v_{p_j}\}$ is locally uniformly upper bounded in $\wi W$ and $\frac{1}{A_{p_j}}\,\wi v_{p_j}\to\wi\psi$ in $L^1_{loc}(\wi W,\Omega^n)$. This proves our claim.

Since $\Sigma$ is compact, we infer by the above that there exists a compact set $K\subset X$ such that $Y\subset\inte K$ and $\frac{1}{A_p}\log P_p\circ\pi\to0$ in $L^1(\pi^{-1}(K),\Omega^n)$. Then
\[\frac{1}{A_p}\int_K|\log P_p|\,\omega^n=\frac{1}{A_p}\int_{\pi^{-1}(K)}
|\log P_p\circ\pi|\,\pi^\star\omega^n\leq\frac{1}{A_p}\int_{\pi^{-1}(K)}|\log P_p\circ\pi|\,
\Omega^n\to0\]
as $p\to\infty$, and the proof is finished.
\end{proof}

\section{Applications}\label{S:appl}

In the case of the sequence of powers of a single line bundle, 
Theorem \ref{T:mt1} yields the following generalization of 
Tian's theorem to the setting of big line bundles on Moishezon spaces:

\begin{Theorem}\label{T:mt2}
Let $X$ be a compact, reduced, irreducible, normal complex space 
of dimension $n$ and $(L,h)$ be a singular Hermitian holomorphic 
line bundle on $X$ such that $c_1(L,h)\geq\varepsilon\omega$, 
where $\varepsilon>0$ is a constant and $\omega$ is a 
Hermitian form on $X$. If $P_p,\gamma_p$ are the Bergman 
kernel function and Fubini-Study current of 
$H^0_{(2)}(X,L^p,h^p,\omega^n)$ then, as $p\to\infty$,
\[\frac{1}{p}\,\log P_p\to 0 \text{ in } 
L^1(X,\omega^n),\;\frac{1}{p}\,\gamma_p\to c_1(L,h) 
\text{ weakly on } X.\]
\end{Theorem}

\begin{proof}
If $\|c_1(L,h)\|=\int_Xc_1(L,h)\wedge\omega^{n-1}$, 
then the assumptions (A)-(B) hold with 
\[a_p=p\,\varepsilon,\;A_p=p\,\|c_1(L,h)\|,\;T_0=c_1(L,h)/\|c_1(L,h)\|.\] 
\end{proof}

Recall that a K\"ahler current is a positive closed current $T$ 
of bidegree $(1,1)$ on $X$ such that $T\geq\varepsilon\omega$ 
for some constant $\varepsilon>0$. 
Let $(L,h)$ be a singular Hermitian holomorphic line bundle on $X$
with positive curvature current $c_1(L,h)\geq0$, and such that $L$ 
has a singular Hermitian metric $h_0$ whose curvature is a K\"ahler current. 
As in \cite[Corollary 5.2]{CMM17}, Theorem \ref{T:mt1} 
can be applied to the sequence of line bundles 
$(L^p,h^{p-n_p}\otimes h_0^{n_p})$, where 
$n_p\in\N$ and $n_p\to\infty$, $n_p/p\to0$ as $p\to\infty$. 
One can also apply Theorem \ref{T:mt1}
to the sequence of tensor products of powers of several line bundles 
as in \cite[Corollary 5.11]{CMM17}. 
We refer to \cite[Section 5]{CMM17} for the details.

\smallskip

Let us consider now the special case when $X$ is smooth, i.e., 
a connected compact complex manifold of dimension $n$. 
If $X$ is assumed to be K\"ahler then the domination condition 
\eqref{e:domin0} is not needed as one can work directly on $X$ 
without the use of a modification $\pi$. 
More precisely, in \cite{CMM17} we proved the following:

\begin{Theorem}{\cite[Theorem 1.2]{CMM17}}\label{T:mtK} 
Let $(X,\omega)$ be a compact K\"ahler manifold of dimension $n$
and $(L_p,h_p)$, $p\geq1$, be a sequence of singular Hermitian
holomorphic line bundles on $X$ which satisfy 
$c_1(L_p,h_p)\geq a_p\,\omega$, where $a_p>0$ and 
$a_p\to\infty$. If $P_p,\gamma_p$ are the Bergman kernel function 
and Fubini-Study current of $H^0_{(2)}(X,L_p)$, 
and if $A_p=\int_Xc_1(L_p,h_p)\wedge\omega^{n-1}$, 
then $\frac{1}{A_p}\log P_p\to 0$ in $L^1(X,\omega^n)$ 
and $\frac{1}{A_p}\,(\gamma_p-c_1(L_p,h_p))\to 0$ weakly on $X$.
\end{Theorem}

However, if $X$ is a Moishezon manifold which is not K\"ahler, 
and hence not projective, we still have to use in our proof of 
Theorem \ref{T:mt1} the modification $\pi:\wi X\to X$ provided 
in Theorem \ref{T:Moi}. So we have to require the domination 
condition \eqref{e:domin0} in assumption (B).

\smallskip

One of the main applications of Tian's theorem is to the study 
of the asymptotic distribution of the zeros of random sequences 
of sections in $H^0(Z,L^p)$ as $p\to\infty$, 
where $L$ is a holomorphic line bundle over a compact 
complex manifold $Z$. 
This started with the pioneering work of 
Shiffman and Zelditch \cite{ShZ99} in the case of 
a positive line bundle $(L,h)$ over a projective manifold $Z$ 
(see also \cite{ShZ08,Sh08}). It is shown in \cite{ShZ99} 
that for almost all sequences 
$\{\sigma_p\in H^0(Z,L^p)\}_{p\geq1}$, 
with respect to the spherical measure, 
one has that $\frac{1}{p}\,[\sigma_p=0]\to c_1(L,h)$ 
weakly on $Z$, where $[s=0]$ denotes the current of 
integration over the zero divisor of a holomorphic section $s$.
In the case of singular Hermitian holomorphic line bundles 
we proved that similar results hold in different contexts \cite{CM15,CM13,CMN16,CMN18}. 

The study of the asymptotic distribution of zeros of random sections in the Bergman spaces $H^0_{(2)}(X,L_p)$ for an arbitrary sequence of singular Hermitian holomorphic line bundles $(L_p,h_p)$ over a compact normal K\"ahler space $X$ was pursued in \cite{CMM17,BCM}. In particular we considered in \cite[Theorem 1.1]{BCM} very general probability measures on the spaces $H^0_{(2)}(X,L_p)$, as follows. We identify the spaces $H^0_{(2)}(X,L_p)\simeq \C^{d_p}$ using fixed orthonormal bases $S_1^p,\dots,S_{d_p}^p$ and we endow them with probability measures $\sigma_p$ such that the following holds:

\medskip

(C) There exist a constant $\nu\geq1$ and for every $p\geq1$ constants $C_p>0$ such that  
\[\int_{\C^{d_p}}\big|\log|\langle a,u\rangle|\,\big|^\nu\,d\sigma_p(a)
\leq C_p\,,\,\text{for any $u\in\C^{d_p}$ with $\|u\|=1$}\,.\] 

\smallskip

Note that \cite[Theorem 1.1]{BCM} holds in our present context. Indeed, we can apply the general equidistribution result \cite[Theorem 4.1]{BCM} together with Theorem \ref{T:mt1}. We recall one of its 
assertions here.

\begin{Theorem}\label{T:zeros}
Assume that $X,\omega,(L_p,h_p),\sigma_p$ 
verify (A),\,(B),\,(C) and consider the product probability space
\[(\mathcal{H},\sigma)=\left(\prod_{p=1}^\infty H^0_{(2)}(X,L_p),\prod_{p=1}^\infty\sigma_p\right).\]
If $\sum_{p=1}^{\infty}C_pA_p^{-\nu}<\infty$ then for $\sigma$-a.e.\ sequence $\{s_p\}\in\mathcal{H}$ we have, as $p\to\infty$,
\[\frac{1}{A_p}\log|s_p|_{h_p}\to0 \text{ in } L^1(X,\omega^n),\;
\frac{1}{A_p}\big([s_p=0]-c_1(L_p,h_p)\big) \to0 \text{ weakly on } X.\]
\end{Theorem}

We refer to \cite{BCM,BCHM} for general classes of 
measures $\sigma_p$ that satisfy condition (C), including Gaussians, 
Fubini-Study volumes, and area measure of spheres. 
Note that if the measures $\sigma_p$ verify (C) with constants 
$C_p=\Gamma_\nu$ independent of $p$ 
(like the Gaussians and the Fubini-Study volumes) 
then the hypothesis of Theorem \ref{T:zeros} becomes 
$\sum_{p=1}^{\infty}A_p^{-\nu}<\infty$.

\smallskip

We close the paper with some remarks of Moishezon manifolds.
By a theorem of Moishezon, a Moishezon manifold
is projective if and only if it carries a K\"ahler metric, see 
\cite{Moi66} and \cite[Theorem 2.2.26]{MM07}.
Moreover, any Moishezon manifold of dimension two
is projective, by Theorem \ref{T:Moi}. Indeed, in dimension two
we can blow up only points and the blow-up $\widehat{X}$
at a point of a compact manifold
$X$ is projective if and only if $X$ is projective.
Hence non-projective Moishezon manifolds
have dimension greater than two.
The first example of this kind was obtained by Hironaka
in his thesis (1961) and is described in \cite[Appendix B, Example 3.4.1]{Ha77}.
It's a manifold which contains a curve which is homologous to zero,
which is impossible on a K\"ahler manifold.
Further examples can be found in \cite{An99,BV,Ko91,Pet95}, 
see also \cite[Section 2.3.4]{MM07}.



\begin{thebibliography}{XXXXX}



\bibitem[A]{An99} M.\ Andreatta,
{\em Moishezon manifolds}, 
Math.\ Z.\ \textbf{230} (1999), no.\ 4, 713--726.

\bibitem[BCHM]{BCHM} T.\ Bayraktar, D.\ Coman, H.\ Herrmann and G.\ Marinescu,
{\em A survey on zeros of random holomorphic sections}, 
Dolomites Res.\ Notes Approx.\ \textbf{11} (2018), Special Issue Norm Levenberg, 1--19. 

\bibitem[BCM]{BCM} T.\ Bayraktar, D.\ Coman and G.\ Marinescu, 
{\em Universality results for zeros of random holomorphic sections}, 
Trans.\ Amer.\ Math.\ Soc.\ \textbf{373} (2020), no.\ 6, 3765--3791.

\bibitem[BCMN]{BCMN} T.\ Bayraktar, D.\ Coman, G.\ Marinescu and V.-A.\ Nguy\^en,
{\em Zeros of random holomorphic sections of big line bundles with continuous metrics}, 
AMS Contemporary Mathematics volume in memory of Steve Zelditch, to appear,
arXiv: 2404.08116.

\bibitem[BV]{BV} L.\ Bonavero and C.\ Voisin,
{\em Sch\'emas de Fano et vari\'et\'es de Moishezon},
C.\ R.\ Acad.\ Sci.\ Paris S\'er.\ I Math.\ \textbf{323} (1996), no.\ 9, 1019--1024.

\bibitem[BM]{BM97} E.\ Bierstone and P.\ Milman, 
{\em Canonical desingularization in characteristic zero by blowing up the maximum strata of a local invariant},
Invent. Math. {\bf 128} (1997), 207--302.





\bibitem[CM1]{CM15} D.\ Coman and G.\ Marinescu, 
{\em Equidistribution results for singular metrics on line bundles}, 
Ann.\ Sci.\ \'Ecole Norm.\ Sup\'er.\ (4) {\bf 48} (2015), 497--536.

\bibitem[CM2]{CM13} D.\ Coman and G.\ Marinescu,
{\em Convergence of Fubini-Study currents for orbifold line bundles}, 
Internat.\ J.\ Math.\ {\bf 24} (2013), 1350051, 27 pp.

\bibitem[CMM]{CMM17}  D.\ Coman,  X.\ Ma and G.\ Marinescu, 
{\em Equidistribution for sequences of line bundles on normal K\"ahler spaces},  
Geom.\ Topol.\ {\bf 21} (2017), 923--962.

\bibitem[CMN1]{CMN16} D.\ Coman, G.\ Marinescu and V.-A.\ Nguy\^en,
{\em H\"older singular metrics on big line bundles and equidistribution},  
Int.\ Math.\ Res.\ Notices {\bf 2016}, no. 16, 5048--5075.

\bibitem[CMN2]{CMN18} D.\ Coman, G.\ Marinescu and V.-A.\ Nguy\^en,
{\em Approximation and  equidistribution results for pseudo-effective line bundles}, 
J.\ Math.\ Pures Appl.\ (9) {\bf 115} (2018), 218--236.

\bibitem[D1]{D82} J.-P.\ Demailly, 
{\em Estimations $L^2$ pour l'op\'erateur $\overline\partial$ d'un fibr\'e holomorphe 
semipositif au--dessus d'une vari\'et\'e k\"ahl\'erienne compl\`ete}, 
Ann.\ Sci.\ \'Ecole Norm.\ Sup.\ (4) {\bf 15} (1982), no.\ 3, 457--511.

\bibitem[D2]{D85} J.-P.\ Demailly, 
{\em Mesures de Monge-Amp\`ere et caract\'erisation 
g\'eom\'etrique des vari\'et\'es alg\'ebriques affines}, 
M\'em.\ Soc.\ Math.\ France (N.S.) No.\ 19 (1985), 1--125.

\bibitem[D3]{D90} J.-P.\ Demailly,
{\em Singular Hermitian metrics on positive line bundles}, 
in {\em  Complex algebraic varieties (Bayreuth, 1990)}, 
Lecture Notes in Math.\ 1507, Springer, Berlin, 1992, 87--104.

\bibitem[D4]{D92} J.-P.\ Demailly, 
{\em Regularization of closed positive currents and intersection theory},  
J.\ Algebraic Geom.\ {\bf 1} (1992), 361--409.

\bibitem[D5]{D93} J.-P.\ Demailly, 
{\em A numerical criterion for very ample line bundles}, 
J.\ Differential Geom.\ {\bf 37} (1993), 323--374.

\bibitem[DS]{DS10} T.-C.\ Dinh and N.\ Sibony,
{\em Dynamics in several complex variables: 
endomorphisms of projective spaces and polynomial-like mappings}, 
{\em Holomorphic dynamical systems}, 165--294,
Lecture Notes in Math.\ {\bf 1998}, Springer, Berlin, 2010.
 
\bibitem[DNS]{DNS08} T.-C.\ Dinh, V.-A.\ Nguy\^ en and N.\ Sibony, 
{\em Dynamics of horizontal-like maps in higher dimension}, 
Adv.\ Math.\ {\bf 219} (2008), 1689--1721.



\bibitem[GM]{GM06} C. Grant Melles and P. Milman,
{\em Classical Poincar\'e metric pulled back off singularities using
a Chow-type theorem and desingularization}, 
Ann.\ Fac.\ Sci.\ Toulouse Math.\ (6) {\bf 15} (2006), 689--771.






\bibitem[Ha]{Ha77}
R.\ Hartshorne,
{\em Algebraic geometry},
Grad.\ Texts in Math., No.\ 52,
Springer-Verlag, New York-Heidelberg, 1977. xvi+496 pp.

\bibitem[Ho]{Ho} L.\ H\"ormander,
{\em Notions of Convexity},  Reprint of the 1994 edition, 
Modern Birkh\"auser Classics, Basel, Birkh\"auser, 2007, viii, 414 pp.


\bibitem[Ko]{Ko91}
J.\ Koll\'{a}r,
\emph{Flips, flops, minimal models, etc.},
Surveys in differential geometry ({C}ambridge, {MA}, 1990),
113--199,
{Lehigh Univ., Bethlehem, PA},
 {1991}.



\bibitem[MM]{MM07} X.\ Ma and G.\ Marinescu, 
{\em Holomorphic Morse Inequalities and Bergman Kernels}, 
Progress in Math., vol.\ 254, Birkh\"auser, Basel, 2007, xiii, 422 pp.

\bibitem[Mo]{Moi66} B.\ G.\ Moishezon,
{\em On $n$-dimensional compact complex manifolds having $n$ algebraically independent meromorphic functions. I, II, III}, 
Izv.\ Akad.\ Nauk SSSR Ser.\ Mat.\ {\bf 30} (1966), 133--174, 345--386, 621--656. English translation: American Mathematical Society Translations Ser.\ 2, 63, 1967, 51--177.



\bibitem[O]{O87} T.\ Ohsawa, 
{\em Hodge spectral sequence and symmetry on compact K\"ahler spaces},  
Publ.\ Res.\ Inst.\ Math.\ Sci.\ {\bf 23} (1987), 613--625.

\bibitem[OT]{OT87} T.\ Ohsawa and K.\ Takegoshi, 
{\em On the extension of $L^2$ holomorphic functions}, 
Math.\ Z.\ {\bf 195} (1987), 197--204.


\bibitem[Pe]{Pet95} T.\ Peternell,
{\em Moishezon manifolds and rigidity theorems},
Bayreuth.\ Math.\ Schr.\ (1998), no.\ 54, 1--108.


\bibitem[R]{Ru98} W.\ D.\ Ruan, 
{\em Canonical coordinates and Bergman metrics}, 
Comm.\ Anal.\ Geom.\ {\bf 6} (1998), 589--631.


\bibitem[Sh]{Sh08} B.\ Shiffman, 
{\em Convergence of random zeros on complex manifolds}, 
Sci.\ China Ser.\ A {\bf 51} (2008), 707--720.

\bibitem[SZ1]{ShZ99} B.\ Shiffman and S.\ Zelditch, 
{\em Distribution of zeros of random and quantum chaotic sections of positive line bundles}, 
Comm.\ Math.\ Phys.\ {\bf 200} (1999), 661--683.

\bibitem[SZ2]{ShZ08} B.\ Shiffman and S.\ Zelditch, 
{\em  Number variance of random zeros on complex manifolds},
Geom.\ Funct.\ Anal.\ {\bf 18} (2008), 1422--1475.

\bibitem[T]{Ti90} G.\ Tian, 
{\em On a set of polarized K\"ahler metrics on algebraic manifolds}, 
J.\ Differential Geom.\ {\bf 32} (1990), 99--130.

\bibitem[U]{U75} K.\ Ueno,
{\em Classification theory of algebraic varieties and compact complex spaces},
Lecture Notes in Math., Vol.\ 439, Springer-Verlag, Berlin-New York, 1975, xix+278 pp.


\end{thebibliography}
\end{document}